\documentclass[12pt, reqno]{amsart}
\usepackage{amsmath,amssymb,amsthm}
\usepackage[english]{babel}
 2

\usepackage[colorlinks=true,linkcolor=blue,citecolor=blue,pdfpagelabels=false]{hyperref}
\newtheorem{thm}{Theorem}[section]
\newtheorem{lem}[thm]{Lemma}
\newtheorem{prop}[thm]{Proposition}

\newtheorem{defn}[thm]{Definition}
\newcommand{\thmref}[1]{Theorem~\ref{#1}}
\newcommand{\lemref}[1]{Lemma~\ref{#1}}
\newcommand{\rmkref}[1]{Remark~\ref{#1}}
\newcommand{\propref}[1]{Proposition~\ref{#1}}

\theoremstyle{remark}
\newtheorem{rmk}{Remark}[section]

\newenvironment{acknowledgements}{\bigskip\textbf{Acknowledgements.}}{}

\renewcommand{\geq}{\geqslant}
\renewcommand{\leq}{\leqslant}

\newcommand{\medmatrix}[4]{\Bigl(\begin{matrix} #1 \!& \!#2 \\[-4pt] #3 \!&\! #4 \end{matrix}\Bigr)}
\begin{document}

\baselineskip=17pt

\title
{The adjoint map of the Serre derivative and special values of shifted Dirichlet series}
\author[A. Kumar]{Arvind Kumar}
\address{Harish-Chandra Research Institute, HBNI\\ Chhatnag Road,
Jhunsi, Allahabad 211019, India.}
\email{kumararvind@hri.res.in}

\date{\today}

\begin{abstract}
We compute the adjoint of the Serre derivative map with respect to the Petersson scalar product by using existing tools of nearly holomorphic modular forms. The Fourier coefficients of a cusp form of integer weight $k$, constructed using this method, involve special values of certain shifted Dirichlet series associated with a given cusp form $f$ of weight $k+2$. As application, we get an asymptotic bound for the special values of these shifted Dirichlet series and also relate these special values with the Fourier coefficients of $f$. We also give a formula for the Ramanujan tau function in terms of the special values of the shifted Dirichlet series associated to the Ramanujan delta function. 
\end{abstract}


\maketitle
\section{Introduction}
For any positive integer $k\ge 4$, let $M_k(\Gamma)$ (resp. $S_k(\Gamma)$) be the space of modular forms (resp. cusp forms) of weight $k$ for a congruence subgroup $\Gamma$ of $SL_2(\mathbb{Z})$. For even $k\ge 2$, the Eisenstein series of weight $k$ is given by 
$$
E_k(z)=1-\frac{2k}{B_k}\sum_{n=1}^{\infty}\sigma_{k-1}(n)q^n,
$$
where $B_k$ is the $k^{\rm th}$ Bernoulli number, $\sigma_{k-1}(n)=\sum_{d|n}d^{k-1}$,
 $q=e^{2\pi i z}$, and $z$ is in the upper half-plane 
$\mathcal{H}$. For $k\ge 4$ and even, the Eisenstein series $E_k$ is a modular form of 
weight $k$ for $SL_2(\mathbb{Z})$. The Eisenstein series $E_2$ and the derivative $\displaystyle Df:= \frac{1}{2 \pi i}\frac{d}{dz}$ of $f\in M_k(\Gamma)$ do not satisfy the modular property. But by taking  a certain linear combination of $Df$ and $E_2f$, we get a function which transforms like a modular form of weight $k+2$. The underlying map on $M_k(\Gamma)$ is called as the Serre derivative, denoted by $\vartheta_k$ and given in section \ref{serre}. We first show that the Serre derivative is a $\mathbb{C}$-linear map from $S_k(\Gamma)$ to $S_{k+2}(\Gamma)$ which are finite dimensional Hilbert spaces. In this article our purpose is to find the adjoint of the Serre derivative map with respect to the Petersson inner product. It gives construction of a cusp form of weight $k$ with interesting Fourier coefficients, from a given cusp form of weight $k+2$.

Using the properties  of  Poincar{\'e}  series and adjoint of linear maps,  W. Kohnen in \cite{kohnen} constructed the adjoint map of the product map by a  fixed cusp form, with respect to the Petersson scalar product. After Kohnen's work, similar results in various other spaces have been obtained by many mathematicians, since the  Fourier coefficients  of the image  of  a  form  involve  special  values  of  certain shifted Dirichlet series attached to these forms, e.g., generalization to Jacobi forms (see \cite{ckk}, \cite{ajbs} and \cite{sakata}), Siegel modular forms (see \cite{lee} and \cite{ajbs2}), Hilbert modular forms (see \cite{wang}) and half-integral weight modular forms (see \cite{ajkm}). 

In this case also, the Fourier coefficients of the image of $f$ under the adjoint of the Serre derivative map, involve special values of certain shifted Dirichlet series associated with the Fourier coefficients of $f$. As an application we obtain relations among the coefficients of modular forms and special values of these shifted Dirichlet series. Moreover, we also give an asymptotic bound of the special values of this shifted Dirichlet series. Since $E_2$ and the derivative of a modular form are quasimodular forms (introduced by Kaneko and Zagier \cite{kaza}), they do not satisfy the modular transformation property. Hence, it is not possible to define the Petersson inner product in the usual way for the space of quasimodular forms. However, there is an isomorphism between the space of quasimodular forms and the space of nearly holomorphic modular forms and we can define the Petersson inner product in the space of nearly holomorphic modular forms. Therefore, sometimes it is convenient to switch our problems from quasimodular forms to nearly holomorphic modular forms and vice versa. By the means of Maass-Shimura operator $R_k$ and $E_2^*$ (see section \ref{ref} for notations), we first transform the definition of the Serre derivative in the context of nearly holomorphic modular forms and then we compute the adjoint map explicitly. We emphasize that our proof can be carried out in the setting of half-integral weight forms to construct cusp forms of half-integer weight.  

\section{Basic tools and notations}\label{ref}
Let $\mu_{\Gamma}$ denote the index of $\Gamma$ in $SL_2(\mathbb{Z})$. Unless otherwise stated we assume that $k \in \mathbb{Z}$. For $\gamma = \medmatrix a b c d \in GL_2^+(\mathbb{Q})$, we define the slash operator as follows:
  $$f|_k \gamma (z):= ({\rm det} \gamma)^{k/2} (cz+d)^{-k}f(\gamma z),\;{\rm{where}}\; \gamma z = \dfrac{az+b}{cz+d}.$$
We denote by $M_k(\Gamma, \chi)$ (resp. $S_k(\Gamma, \chi)$) the space of modular forms (resp. cusp forms) of weight $k$ for the congruence subgroup $\Gamma$ of level $N$ and Dirichlet character $\chi$ (mod $N$). We write $M_k(N)$ (resp. $S_k(N)$) for the corresponding spaces if $\Gamma = \Gamma_0(N)$ and $\chi$ is the principal character.

Let $f, g \in {M}_k(\Gamma)$ be such that the product $fg$ is a cusp form. Write $z=x+iy$, then the Petersson inner product is defined by 
\begin{equation}\label{ip3}
\langle f,g \rangle_{} := \frac{1}{\mu_{\Gamma}} \int\limits_{z \in \Gamma \setminus \mathcal{H}} 
f(z) \overline{g(z)}y^k \frac{dx dy}{y^2}.
\end{equation}
Let $k$ and $ n$ be positive integers. The $n^{\rm 
th}$ Poincar{\'e} series of weight $k$ for a congruence subgroup $\Gamma$ is defined by
\begin{eqnarray}
P_{k, n} (z):=\sum_{\gamma \in \Gamma_{\infty}\setminus \Gamma} e^{2 \pi i nz}|_k \gamma,
\end{eqnarray} 
where $ \Gamma_{\infty}:=\left\{ \pm \medmatrix 1 t 0 1 | t  \in \mathbb Z  \right\}.$\\
It is well known that $P_{k, n} \in S_k(\Gamma)$ for $k >2$ and it is characterized by the following property which is known as the Petersson coefficient formula:
\begin{lem}\label{poincare-lemma}
Let $ f \in S_k(\Gamma)$ with Fourier expansion $f(z)=\sum\limits_{m=1}^\infty a(m) q^m .$ Then 
\begin{equation*}
\langle f,~  P_{k, n} \rangle= \alpha_{k, n} a(n),
\end{equation*}
where, $\displaystyle{\alpha_{k, n}=\frac{\Gamma(k-1)} {(4\pi n)^{k-1}}.}$
\end{lem}
The following familiar result tells about the growth of Fourier coefficients of a modular form in which the first statement is trivial and the second is due to P. Deligne \cite{Deligne}:
\begin{prop}\label{bound}
If $f\in M_k (\Gamma,\chi)$ with Fourier coefficients $a(n),$ then 
$
a(n) \ll n^{k-1+\epsilon}, 
$\\ 
and moreover, if $f$ is a cusp form, then
$
a(n) \ll n^{\frac{k-1}{2}+\epsilon}$, for any $\epsilon >0$. 
\end{prop}
\subsection{Nearly holomorphic modular forms}
For the convenience of the reader we repeat the relevant material from \cite{shimura} in our setting without proofs, thus making our exposition self-contained.
\begin{defn}
A nearly holomorphic modular form $f$ of weight $k$ and depth $\le p$ for  $\Gamma$ is a polynomial in $\frac{1}{\Im(z)}$ of degree $\le p$ whose coefficients are holomorphic functions on 
$\mathcal{H}$ with moderate growth such that
$$
f|_k \gamma (z)=f(z),
$$
for any $\gamma = \left(\!\!\begin{array}{ll}a&b\\ c&d\\ \end{array}\!\!\right)\in \Gamma$ and $z\in \mathcal{H}$, where $\Im(z)$ is the imaginary part of $z$.
\end{defn}
Let $\widehat{M}_k^{\le p}(\Gamma)$ denote the space of nearly holomorphic modular forms of 
weight $k$ and depth $\le p$ for $\Gamma$. We denote by $\widehat{M}_k(\Gamma)=\bigcup_p\widehat{M}_k^{\le p}(\Gamma)$ the space of all nearly holomorphic modular forms of weight $k$. Note that 
\begin{equation}
E_2^*(z)=E_2(z)-\frac{3}{\pi \Im(z)}
\end{equation} 
is a nearly holomorphic modular form of weight $2$
for the group $SL_2(\mathbb{Z})$.

\begin{defn}
Let $f\in \widehat{M}_k(\Gamma)$. Then $f$ is called a
\begin{itemize}

\item
{\bf{rapidly decreasing function}} at every cusp of $\Gamma$ if
for each $\alpha \in SL_2 (\mathbb {Q})$ and positive real number $c$, there exist positive constants $A$ and $B$ depending on $f$, $\alpha$ and $c$ such that
$$ | \Im (\alpha z)^{k/2} f(\alpha z)| < Ay^{-c}~~~ {\mbox  if}~ y = \Im(z) > B.$$
\item
{\bf{slowly increasing function}} at every cusp of $\Gamma$ if 
for each $\alpha \in SL_2 (\mathbb{Q})$ there exist positive constants $A$, $B$ and $c$ depending on $f$ and $\alpha$ such that
$$ | \Im( \alpha z )^{k/2} f (\alpha z)| < A y^c ~~~{\mbox{if}}~ y = \Im(z) > B. $$
\end{itemize}
\end{defn}

\begin{rmk}
For example, a modular form is slowly increasing and a cusp form is rapidly decreasing function. Moreover, the product of a rapidly decreasing function with any nearly holomorphic modular form provides a rapidly decreasing form.

\end{rmk}

Let $f, g \in \widehat{M}_k(\Gamma)$ be such that the product $fg$ is a rapidly decreasing function. Write $z=x+iy$, then the inner product is defined by 
\begin{equation}\label{ip}
\langle f,g \rangle_{} := \frac{1}{\mu_{\Gamma}} \int\limits_{z \in \Gamma \setminus \mathcal{H}} 
f(z) \overline{g(z)}y^k \frac{dx dy}{y^2}.
\end{equation} 
By abuse of notation, we used the same symbol here for Petersson inner product as in the case of modular forms given in \eqref{ip3}. The integral is convergent because of the hypothesis and hence the inner product is well defined.

\begin{defn}
The Maass-Shimura operator $R_k$ on $f\in \widehat{M}_k(\Gamma)$ is defined by
$$
R_k(f)=\frac{1}{2\pi i}\left(\frac{k}{2i\Im(z)}+\frac{\partial}{\partial z}\right)f(z).
$$
\end{defn}
\noindent
The operator $R_k$ takes $\widehat{M}_k(\Gamma)$ into $\widehat{M}_{k+2}(\Gamma)$, so sometimes it is called Maass raising operator. There is another operator 
$L_k:= -y^2  \frac{\partial}{\partial {\overline{z}}}: \widehat{M}_{k+2}(\Gamma) \rightarrow \widehat{M}_k(\Gamma)$, known as Maass lowering operator, which annihilates any holomorphic function. In \cite[Theorem 6.8]{shimura}, It was shown that the operator $L_k$ is the adjoint of $R_k$ with respect to the inner product \eqref{ip}, as long as the product of the functions is rapidly decreasing (see also \cite[Lemma 4.2]{brulnm}). We now state an interesting application of this observation which plays a crucial role in the proof of our main result.
\begin{lem}\label{ip0}
Let $\Gamma$ be a congruence subgroup and let $f \in S_{k+2}(\Gamma)$. Then 
$\langle f, R_{k}g \rangle = 0$ for any $g \in M_{k}(\Gamma)$.
\end{lem}

\subsection{A shifted Dirichlet series}
Let $f(z)=\sum\limits_{n=0}^\infty a(n) q^n$ and $g(z)= \sum\limits_{n=0}^\infty b(n) q^n$. For $m \geq 0$, define a shifted Dirichlet series of Rankin type as in \cite{kohnen} by
\begin{eqnarray}
L_{f,g, m}(s)= \sum\limits_{n\geq 1}\frac{a(n+m) \overline{b(n)}}{(n+m)^s}.
\end{eqnarray}
If the coefficients $a(n)$ and $b(n)$ satisfy an appropriate bound, then $L_{f,g, m}(s)$ converges absolutely in some half-plane.\\
For $f \in S_k(\Gamma, \chi)$ and a non negative integer $m$, consider a shifted Dirichlet series of Rankin type associated with $f$ and $E_2$ denoted by $L_{f,m}(s)$ and defined by
\begin{equation}
L_{f,m}(s):=-\frac{1}{24}L_{f, E_2, m}(s)=\sum\limits_{n\geq 1}\frac{a(n+m) \sigma(n)}{(n+m)^s}.
\end{equation}
Then by \propref{bound}, it is absolutely convergent for Re$(s)>\frac{k+3}{2}$. It can be shown that $L_{f, m}(s)$ has a meromorphic continuation to $\mathbb{C}$ (compare \cite[section 2-5]{hohu}). A slightly different shifted Dirichlet series of this kind associated with two modular forms was first introduced by Selberg in \cite{selberg}. Recently, in \cite{hohu},  J. Hoffstein and T. A. Hulse rigorously investigated the meromorphic continuation of a variant of Selberg's shifted Dirichlet series and multiple shifted Dirichlet series. In \cite{meon}, M. H. Mertens and K. Ono proved that certain special values of symmetrized sum of such functions involve as the coefficients of sum of mixed mock modular forms and quasimodular forms. The $p$-adic properties of these were furthermore studied in \cite{bmo}.

\section{Serre derivative}\label{serre}
It is well known \cite[Proposition~2.11]{ono} that for a positive integer $k$ and $f\in M_k(\Gamma)$ the function 
\begin{equation}\label{serredef}
\vartheta_kf := Df-\frac{k}{12}E_2f
\end{equation} 
is a modular form of weight $k+2$ for $\Gamma$. The weight $k$ operator $\vartheta_k$ defined by \eqref{serredef} is called the \textit{Serre derivative} (or sometimes \textit {the Ramanujan-Serre differential operator}). 
It is an interesting and useful operator because it defines an operator on modular forms for any congruence subgroup with character and also preserves cusp forms. 
\begin{thm}\label{serretrans}
Let $k$ be a non negative integer. Then the weight $k$-operator $\vartheta_k$ maps $M_k(\Gamma, \chi)$ to $M_{k+2}(\Gamma, \chi)$ and $S_k(\Gamma, \chi)$ to $S_{k+2}(\Gamma, \chi)$. In particular $\vartheta_k$ maps $S_k(N)$ to $S_{k+2}(N)$.
\end{thm}

\begin{rmk}
We observe that, the Serre derivative can also be written in the form
\begin{eqnarray}
\vartheta_k(f)= R_k f-\frac{k}{12}E_2^{*} f.
\end{eqnarray}
This form is quite useful while computing the Petersson inner product as $R_k f$ and $E_2^{*}f$ are nearly holomorphic modular forms, where the inner product is defined by \eqref{ip}, provided $f$ is a cusp form.
\end{rmk}

\begin{rmk}\label{serrehalf}
Similar to \eqref{serredef}, we can define the weight $\frac{k}{2}$-operator   $\vartheta_{k/2}$ for an odd positive integer $k$. Then one can easily see that $\vartheta_{k/2}$ maps a modular form (resp. cusp form) of weight $\frac{k}{2}$ to modular form (resp. cusp form) of weight $\frac{k}{2}+2$.
\end{rmk}
\section{Main Theorem}
From \thmref{serretrans} we know that $\vartheta_k : S_k(\Gamma) \rightarrow S_{k+2}(\Gamma)$ is a $\mathbb{C}-$linear map of finite dimensional Hilbert space and hence has an adjoint map $\vartheta_k^{*} : S_{k+2}(\Gamma) \rightarrow S_k(\Gamma)$, such that
\begin{eqnarray*}
\langle \vartheta_k^{*}  f, g \rangle = \langle  f, \vartheta_k g\rangle,~~~~ \forall ~ f \in S_{k+2}(\Gamma), ~ g \in S_k(\Gamma).
\end{eqnarray*}
In  the  main  result  we  exhibit  the  Fourier  coefficients  of $\vartheta_k^{*}  f$ for $f \in S_{k+2}(\Gamma)$. Its $m^{\rm th}$ Fourier coefficient involves special values of the shifted Dirichlet series $L_{f, m}(s)$. Now we shall state the main theorem of this article.
\begin{thm}\label{main}
Let k $\geq$ 2. The image of any function $f(z) = \sum\limits_{n \geq 1} a(n)q^n \in S_{k+2}(\Gamma)$ under $\vartheta_k^{*}$ is given by 
$$ \vartheta_k^{*}f (z)= \sum_{m \geq 1}c(m) q^m, $$
where $$c(m) = \frac{1}{\mu_{\Gamma}} \frac{k(k-1)m^{k-1}}{{(4 \pi)}^2} \left[ \frac{(m- \frac{k}{12})}{m^{k+1}}a(m) + 2kL_{f, m}(k+1) \right].$$
\end{thm}

\section{Applications}
\subsection{An asymptotic bound for \texorpdfstring{$L_{f, m}(k+1)$}{Lg}}
Let $f \in S_{k+2}(\Gamma)$ and $\Gamma$ be a congruence subgroup of level $N$.
From \thmref{main}, we can write
\begin{equation*}
L_{f, m}(k+1)= \frac{1}{2k} \left[  \frac{ \mu_{\Gamma} {(4 \pi)}^2}{k(k-1)} \frac{1}{m^{k-1}}c(m)  -\frac{(m- \frac{k}{12})}{m^{k+1}}a(m) \right].
\end{equation*}
Here, $c(m)$ is the $m^{\rm th}$ Fourier coefficient of $\vartheta_k^*f$ which is a cusp form of weight $k$. Hence, in view of \propref{bound}, a direct calculation gives
\begin{equation}
L_{f, m}(k+1)\ll m^{\frac{1-k}{2}},
\end{equation}
where the implied constant depends on $f$.
\subsection{Values of \texorpdfstring{$L_{f, m}(k+1)$}{Lg} in terms of the Fourier coefficients}
Let $k \geq 2$ and $\Gamma$ be a congruence subgroup for which $S_k(\Gamma)$ is a 
one-dimensional space; we denote a generator of $S_k(\Gamma)$ by $f(z)$. Then applying \thmref{main}, we get $\vartheta_k^{*}g(z)= \alpha_g f(z)$ for any $g \in S_{k+2}(\Gamma)$, where $\alpha_g$ is a constant. Now equating the $m^{\rm th}$ Fourier coefficients both the sides,  we get a relation among the special values of the shifted Dirichlet series associated with $g$ and the Fourier coefficients of $f$. In the following, we illustrate this with one example.

From now on, $\Delta_{k, N}$ will denote the unique normalized cusp form with Fourier coefficients $\tau_{k ,N}(n)$ in the one dimensional space $S_k(N)$. 
Note that $\Delta_{12, 1}(z)= \Delta(z)$, whose Fourier coefficients $\tau (n)$, the Ramanujan tau function. For a positive integer $t$, we introduce the $V$-operator acting on a function $f$ (defined on $\mathbb{C}$) by 
\begin{equation*}\label{Vop}
V_tf(z):= f(tz).
\end{equation*}
It is known that $V_t$ is a linear operator from $S_k(N)$ into $S_k(Nt)$. Note that 
$S_{10}(2)=\mathbb{C}\Delta_{10, 2}(z)$ and $S_{12}(2)=\mathbb{C}\Delta (z) \oplus \mathbb{C} V_2\Delta(z)$. Now
considering the map $\vartheta_{10}: S_{10}(2) \rightarrow S_{12}(2)$, a direct computation shows that
\begin{equation}\label{exp10}
\vartheta_{10} \Delta_{10, 2}(z) 
= \frac{1}{6} \Delta (z)+ \frac{128}{3} V_2\Delta(z)
\end{equation} 
Let 
$ \vartheta_{10}^{*} \Delta(z) = \alpha \Delta_{10, 2}(z)$  and $\vartheta_{10}^{*} V_2\Delta(z) = \beta \Delta_{10,2}(z)$,  for some $ \alpha, \beta \in \mathbb{C}.$
By using the property of the adjoint map and \eqref{exp10}, we have
\begin{align}\label{alpha}
\alpha \Vert\Delta_{10, 2}\Vert^2
&= \langle \alpha \Delta_{10, 2} , \Delta_{10, 2} \rangle
= \langle \vartheta_{10}^{*} \Delta, \Delta_{10, 2} \rangle \notag \\
&= \langle  \Delta, \vartheta_{10} \Delta_{10, 2} \rangle 
= \langle  \Delta, \frac{1}{6} \Delta+ \frac{128}{3} V_2\Delta \rangle \notag\\
 &=\frac{1}{6} \Vert\Delta\Vert^2+ \frac{128}{3}\langle \Delta, V_2\Delta\rangle.  
\end{align}
\noindent
Similarly,  
\begin{equation}\label{beta}
\beta \Vert\Delta_{10, 2}\Vert^2 =\frac{128}{3} \Vert V_2\Delta\Vert^2+ \frac{1}{6}\langle \Delta, V_2 \Delta \rangle.~~~~~~~~~~~~~
\end{equation}
From {\cite[Eq. 49]{chi}}, we know that
$ \langle \Delta, V_2\Delta\rangle = -\frac{1}{256}\Vert\Delta\Vert^2$.
Using it in expression \eqref{alpha}, we get $\alpha = 0$, which gives
\begin{equation}\label{adj10}
\vartheta_{10}^{*} \Delta(z) = 0.
\end{equation}
Now applying \thmref{main}, we have
\begin{eqnarray*}
\frac{(m- \frac{10}{12})}{m^{11}}\tau(m) + 20L_{\Delta, m}(11) =0
\end{eqnarray*}

\begin{equation}\label{tau}
\tau(m) = \frac{-20m^{11}}{(m- \frac{5}{6})} L_{\Delta, m}(11).
\end{equation}
From {\cite[Proposition 46]{koblitz}}, we get
$$ \Vert V_2\Delta \Vert^2 = 2^{-12} \Vert \Delta \Vert^2$$
and \eqref{beta} gives
\begin{equation}
\beta =\frac{5}{2^9} \frac{\Vert\Delta\Vert^2}{\Vert\Delta_{10, 2}\Vert^2}.
\end{equation}
Using \thmref{main} in the expression 
$\Delta_{10,2}(z)=\frac{1}{\beta}\vartheta_{10}^* V_2\Delta(z)$, we get 
\begin{equation}\label{tau2}
\tau_{10,2}(m)=\frac{15 m^9}{8 \beta \pi^2 }\left[ \frac{(m- \frac{10}{12})}{m^{11}} \tau\left(\frac{m}{2}\right) + 20L_{V_2\Delta, m}(11) \right],
\end{equation}
where $\tau (n) =0$ if $n$ is not an integer.\\
Therefore, for odd $m$, we have
\begin{align*}
\tau_{10,2}(m)&=\frac{3840 m^9}{\pi^2 }  \frac{\Vert\Delta_{10, 2}\Vert^2} {\Vert\Delta\Vert^2}  L_{V_2\Delta, m}(11),
\end{align*}
where $\displaystyle{L_{V_2\Delta, m}(11)=\sum_{\substack{n \geq 1 \\ n: {\rm odd}}}}
\frac{\tau(\frac{m+n}{2})\sigma(n)}{(m+n)^{11}}$.
\begin{rmk}\label{taufn}
From \eqref{tau}, we see that for any $m\geq 1$ there exist $n \geq 1$ such that $\tau (m)$ and $\tau(m+n)$ are of opposite sign. In other words, it follows that Ramanujan tau function $\tau(m)$ and $L_{\Delta, m}(11)$ both exhibit infinitely many sign changes. We can also find the values of $L_{\Delta, m}(11)$ for each $m\geq 1$, in particular $L_{\Delta, 1}(11)=-\frac{1}{120}$. Moreover, Lehmer's conjecture is equivalent to non-vanishing of $L_{\Delta, m}(11)$. From \eqref{tau}, we also observe that for each $m \geq 1, L_{\Delta,m}(11) \in \mathbb{Q}$, because the coefficient field of $\Delta$ is $\mathbb{Q}$.
\end{rmk}
In general for any $f \in S_{k+2}(\Gamma)$ and $m \geq 1$, using similar method, we can write $L_{f,m}(k+1)$ as a linear combination of $m^{\rm th}$ Fourier coefficients of $f$ and elements from a fixed basis of $S_k(\Gamma)$. Then analogous observations can be made as in \rmkref{taufn}.

\section{Proof of the main Theorem}
We need the following Lemma to prove the main theorem.
\begin{lem}\label{interchange}
Using the same notation as in \thmref{main}, the following sum of integrals
\begin{equation*}
\sum_{\gamma\in\Gamma_{\infty}\setminus \Gamma} 
 \int_{\Gamma\setminus\mathcal{H}}  
\mid f(z) \overline{E_2^{*}(z)}~\overline{e^{2\pi imz}|_k\gamma}~y^{k+2}\mid ~\frac{dx dy}{y^2}
\end{equation*}
converges. 
\end{lem}
\begin{proof}
Since $f$ is a cusp form of weight $k$, the function $fE_2^{*}$ is a nearly holomorphic modular form of weight $k+2$ and is rapidly decreasing at every cusp. Therefore, for some positive constant M, we have
\begin{center}
$
\mid y^{\frac{k}{2}+1} f(z) {E_2^{*}(z)}\mid ~ \leq M, ~~~\forall ~z\in \mathcal{H}.
$
\end{center}
Changing the variable $z$ to $\gamma^{-1} z$ and using the standard Rankin unfolding argument, the sum in \lemref{interchange} equals to
\begin{align*}
\int_{\Gamma_{\infty}\setminus\mathcal{H}}  
\mid f(z) \overline{E_2^{*}(z)}~\overline{e^{2\pi imz}}~y^{k+2}\mid ~\frac{dx dy}{y^2} 
&= \int_{\Gamma_{\infty}\setminus\mathcal{H}}  
\mid y^{\frac{k}{2}+1} f(z) {E_2^{*}(z)}\mid {e^{- 2\pi my}}~y^{\frac{k}{2}+1} ~\frac{dx dy}{y^2}\\
& \leq M \int_0^{\infty} \int_0^1 e^{- 2\pi my}~y^{\frac{k}{2}-1} ~dx dy = M \frac{\Gamma (\frac{k}{2})}{(2 \pi m)^{\frac{k}{2}}}. 
\end{align*}
\end{proof}
\noindent
\textit{Proof of \thmref{main}.} Since $\vartheta_k^{*} f= \sum\limits_{m \geq 1}c(m) q^m$, by \lemref{poincare-lemma}, we get
\begin{align*}
c(m) &= \frac{(4\pi m)^{k-1}}{\Gamma(k-1)}  \langle \vartheta_k^{*}f, P_{k,m}\rangle 
= \frac{(4\pi m)^{k-1}}{\Gamma(k-1)}  \langle f, \vartheta_k P_{k,m}\rangle.
\end{align*}
By considering the above inner product in the space of nearly holomorphic modular forms and using \lemref{ip0}, we get
\begin{align*}
  \langle f, \vartheta_k P_{k,m}\rangle  & = \langle f, R_k P_{k,m} -\frac{k}{12} E_2^{*} P_{k, m} \rangle \notag = \langle f, R_k P_{k,m} \rangle -\frac{k}{12} \langle f, E_2^{*} P_{k, m} \rangle \\
&= -\frac{k}{12} \langle f, E_2^{*} P_{k, m} \rangle. 
\end{align*}
Hence,
\begin{equation}\label{c(m)}
c(m)= -\frac{k}{12}\frac{(4\pi m)^{k-1}}{\Gamma(k-1)} \langle f, E_2^{*} P_{k, m} \rangle.
\end{equation}
Now consider,
\begin{align*}
\langle f, E_2^{*} P_{k, m} \rangle &= 
\frac{1}{\mu_{\Gamma}} \int_{\Gamma\setminus\mathcal{H}}  
f(z) \overline{E_2^{*}(z) P_{k,m}}~y^{k+2} ~\frac{dx dy}{y^2}\\
&= \frac{1}{\mu_{\Gamma}} \int_{\Gamma\setminus\mathcal{H}}  
 f(z) \overline{E_2^{*}(z)}~\overline{\sum_{\gamma\in\Gamma_{\infty}\setminus \Gamma}e^{2\pi imz}\mid_k\gamma}~y^{k+2} ~\frac{dx dy}{y^2}.
\end{align*}
By \lemref{interchange}, we can interchange summation and integration in the above expression. Using Rankin's unfolding argument, the integral in the above expression can be written as
\begin{align}\label{suminterchange}
 &
\int_{\Gamma_{\infty}\setminus\mathcal{H}} f(z) \overline{E_2^{*}(z)}~\overline{e^{2\pi imz}}~y^{k} ~dx dy \notag \\
&= \int_0^{\infty} \int_0^1 \sum_{s \geq 1} a(s) e^{2 \pi i s (x+iy)} \overline{\left( 1- \frac{3}{\pi y}-24 \sum_{t \geq 1} \sigma(t) e^{2 \pi i t (x+iy)} \right)} \overline{e^{2 \pi i m (x+iy)}} y^k dx dy \notag \\
&= \sum_{s \geq 1} a(s) \int_0^{\infty} \int_0^1  {\left( 1- \frac{3}{\pi y} \right)}e^{-2 \pi y (s+m)}y^k e^{2 \pi i x (s-m)}  dx dy \notag \\
&\hspace{50pt} -24 \sum_{t \geq 1} \sigma(t) \sum_{s \geq 1}  a(s) \int_0^{\infty} \int_0^1  
e^{-2 \pi y(t+m-s)}y^k  e^{2 \pi i x (s-t-m)} dx dy \\
&=  
a(m) \int_0^{\infty} {\left( 1- \frac{3}{\pi y} \right)}e^{-4 \pi m y }y^k dy -24~ \sum_{t \geq 1} a(t+m) \sigma(t) \int_0^{\infty} 
e^{-4 \pi y(t+m)}y^k dy \notag \\
&= \frac{\Gamma(k)}{\pi (4 \pi m)^k}\left(\frac{k}{4m}-3\right) a(m)- 24 \frac{\Gamma(k+1)}{(4 \pi)^{k+1}}
\sum_{t \geq 1} \frac{a(t+m) \sigma(t)}{(t+m)^{k+1}}\notag \\
&= \frac{\Gamma(k)}{(4 \pi)^{k+1}} \left[ \frac{(k-12m)}{m^{k+1}}a(m)-24kL_{f, m}(k+1)\right]. \notag
\end{align}
By \propref{bound}, interchanging the sum and integral in \eqref{suminterchange} is justified. 
Hence $$\langle f, E_2^{*} P_{k, m} \rangle = 
\frac{1}{\mu_{\Gamma}} \frac{\Gamma(k)}{(4 \pi)^{k+1}} \left[ \frac{(k-12m)}{m^{k+1}}a(m)-24kL_{f, m}(k+1)\right]. $$
This proves the theorem.
\begin{rmk}
It is worth pointing out that \lemref{ip0} holds good for forms of half-integral weight. So in view of \rmkref{serrehalf} and using the same technique as in the proof of \thmref{main}, one can explicitly find the map
\begin{eqnarray*}
\vartheta_{k/2}^* : S_{\frac{k}{2}+2}(\Gamma) \longrightarrow S_{\frac{k}{2}}(\Gamma),
\end{eqnarray*}
where $k$ is an odd positive integer and $\Gamma= \Gamma_0(N), N \in 4\mathbb{N}$. It gives a construction of cusp forms of half-integral weight whose coefficients involve special values of shifted Dirichlet series of Rankin type. 
\end{rmk}

\begin{acknowledgements}
I would like to thank Dr. B. Sahu for raising this question and for some useful discussions. I am grateful to my supervisor Prof. B. Ramakrishnan for his constant support, comments and suggestions about the paper. The work has been supported by the SPM research grant of the Council of Scientific and Industrial Research (CSIR), India.
\end{acknowledgements} 

\end{document}